\documentclass[11pt,reqno]{amsart}
\usepackage[T1]{fontenc}
\usepackage{graphicx}
\usepackage{todonotes}
\usepackage{etoolbox}

\usepackage{color}
\definecolor{MyLinkColor}{rgb}{0,0,0.4}
\usepackage{hyperref}

\definecolor{MyLinkColor}{rgb}{0,0,0.4}

\newcommand{\R}{{\mathbb R}}

\newcommand{\Z}{{\mathbb Z}}
\newcommand{\N}{{\mathbb N}}
\newcommand{\s}{\mathbb S}

\newcommand{\V}{\mathcal{V}}

\newcommand{\kL}{\mathcal{L}}

\newcommand{\p}{\partial}

\newtheorem{thm}{Theorem}[section]

\newtheorem{lemma}[thm]{Lemma}

\theoremstyle{remark} 
\newtheorem{rem}[thm]{Remark}

\numberwithin{equation}{section}
 
\makeatletter
\patchcmd{\maketitle}{\@fnsymbol}{\@alph}{}{}  
\makeatother

\title[Flow properties of the PME]{Non-uniform continuity of the semiflow map associated to  the porous medium equation}

\author[B.--V. Matioc]{Bogdan--Vasile Matioc}
\address{Institut f{\"u}r Angewandte Mathematik, Leibniz Universit{\"a}t Hannover, Welfengarten~1, 30167 Hannover, Germany.}
\email{matioc@ifam.uni-hannover.de}

\date{\today}

\subjclass[2010]{35B30, 35G25, 35K65, 76S05}
\keywords{Porous medium equation; Semiflow map; non-uniform continuity.}

\begin{document}

\maketitle

\begin{abstract}
 We prove that the semiflow map associated to the evolution problem for the porous medium equation (PME) is real-an\-a\-lyt\-ic as a function 
 of the initial data in $H^s(\mathbb{S})$, $s>7/2,$ 
 at any fixed positive time, but it is not uniformly continuous.  
More precisely, we construct two sequences of  exact positive solutions of the PME which at initial time converge to zero in $H^s(\mathbb{S})$,
 but such that the limit inferior of the difference of the two sequences is bounded away from zero in $H^s(\mathbb{S})$  at any later time.
\end{abstract}

\section{Introduction and the main result}\label{Sec:1}

We consider herein the evolution problem associated to the one-dimensional  PME
\begin{equation}\label{PME}
 u_t=(uu_x)_x, \qquad \text{$x\in\s$ and $t>0$},
\end{equation}
in the periodic case, that is with $\s$ denoting the unit circle $\s:=\R/(2\pi \Z).$
The  initial condition for \eqref{PME} is  given by the relation
\begin{equation}\label{IC}
 u(0)=u_0.
\end{equation}
The equation \eqref{PME} and its generalization 
\begin{equation}\label{PMEG}
u_t=\Delta u^m, \qquad \text{$t>0$, }
\end{equation}
with $m>1,$ have received in the last decades lots of attention from people working in mathematics and not only.
A  systematic presentation of the mathematical
theory for the PME is presented in the book by V\'azquez \cite{Va07}.
There are nevertheless still many interesting features related to \eqref{PMEG} which have not been studied yet.

For the choice $m=2$   made here, the PME  is a model describing groundwater flows \cite{Bou03} and dates back to Boussinesq's derivation in 1903.
Its two-phase version has been only recently derived in \cite{EMM12} as the lubrication approximation of the Muskat problem.
The question we are interested in  is whether the semiflow map $[(t,u_0)\mapsto u(t;u_0)]$ associated to \eqref{PME}, cf. Theorem \ref{T:1}, 
is uniformly continuous in $H^s(\s)$ as a function of the initial data when keeping the (positive) time fixed.
The uniform continuity of the flow map has been investigated recently in the context of  several hyperbolic models for water waves: 
the Camassa-Holm equation \cite{HK09, HKM10}, the equation for the wave surface corresponding to the Camassa-Holm equation \cite{CL09, DGM13x},
the Euler equations \cite{HM10},  the $b$-equation \cite{Gr_B}, the $\mu$-$b$ equation \cite{LPW13}, the hyperelastic rod equation \cite{Ka10}, the Novikov equation \cite{Gr_N},
the modified Camassa-Holm equation \cite{FL13}, the modified Camassa-Holm system \cite{WL12},   
the answer being always negative.
We should emphasize that all these hyperbolic models can be written as first order  nonlinear equations, the solutions breaking some times in finite time, 
cf. e.g. \cite{CE98, MK98}.
On the other hand, the equation \eqref{PME} is parabolic (degenerate when $u$ becomes zero), of second order,
and it possesses globally defined strong solutions which converge towards flat states\footnote{
Poincar\'e's inequality and parabolic maximum principles ensure that any of the solutions  $u$ of \eqref{PME}  found in  Theorem \ref{T:1} satisfies
\[\frac{d}{dt}\|\Lambda^1(u-[u])\|_{L_2}^2=-2\int_{\s}u(u_x^2+u_{xx}^2)\, dx\leq -2\min u(0)\|\Lambda^1(u-[u])\, dx\|_{L_2}^2\qquad\text{in $(0,\infty),$}\]
with $[u]$ denoting the mean integral value of $u$ over one period.
This implies exponential convergence in $H^1(\s)$.
The principle of linearized stability, cf. e.g. 
\cite[Theorem 9.1.2]{L95}, can be additionally used to prove exponential convergence in stronger Sobolev norms, provided that the initial data 
are close to their mean value in these norms.}.
Moreover, keeping the positive time fixed, the semiflow map  is real-analytic with respect to the initial data, cf. Theorem \ref{T:1}.
Let us also recall that in the setting of non-negative $L_1(\R)-$solutions  with finite mass \cite{KV88}, the semiflow map is in fact a contraction
at each fixed $t>0,$ that is 
\[
\|u(t;u_0)-u(t;v_0)\|_{L_1(\R)}\leq \|u_0-v_0\|_{L_1(\R)}.
\]
Additionally, the semiflow map is  a contraction also with respect to all Wasserstein distances $W_p$, with $p\in[1,\infty],$
\[
W_p(u(t;u_0),u(t;v_0))\leq W_p(u_0,v_0),
\]
cf. e.g. \cite{CGT04, CT03},
the  $W_2$-contractivity in arbitrary space dimensions being established in \cite{CMV06}.
In higher space dimensions $W_p$-contractivity with $p$ large does not hold \cite{V05}. 
For these reasons, the non-uniform continuity property established in Theorem \ref{MT} for the semiflow map associated to \eqref{PME}-\eqref{IC}
is surprising, the more because for the linear correspondent of \eqref{PME}, that is the heat equation,
the semiflow map is a contraction at any fixed positive time, cf. Remark \ref{R:1}.

To establish our result, we first construct two sequences of positive approximate solutions of \eqref{PME}
which  are sufficiently close to the exact solutions of \eqref{PME} defined by the value of the approximate ones at $t=0$.
The sequences of approximate and exact solution are approaching when $n\to\infty$  the  regime where the equation becomes degenerate.
Nevertheless, using parabolic maximum principles for the solution of \eqref{PME} and for its first spatial derivative (this is the reason why $m=2$ is so important), 
commutator estimates, and
interpolation properties of the Sobolev spaces, we show that the difference between the exact solutions that we have found is bounded away, in the limit $n\to\infty$, from zero 
at any positive time, 
although it converges to zero at 
$t=0$.  
As far as we know, this is the first result  that proves, in high order Sobolev spaces, the non-uniform continuity with respect to the initial data  for the semiflow map 
corresponding to a  parabolic evolution equation.

In order to present our main result, let $u(\cdot;u_0)$ denote  the unique strong  solution of \eqref{PME}-\eqref{IC}  associated  to an initial data 
\[u_0\in \mathcal{V}_s:=\{u\in H^s(\s)\,:\, \min_{\s} u>0 \},\] 
whereby for technical reasons we are restricted to considering  $s>7/2$ (see Lemma \ref{L:2}).  
The mapping 
\[\big[[0,\infty)\times \V_s\ni (t,u_0)\mapsto u(t;u_0)\in\V_s\big]\]
defines a global continuous semiflow which is real-analytic in $(0,\infty)\times \V_s$, cf. Theorem \ref{T:1}.
The main result of this paper is the following theorem.

\pagebreak

\begin{thm}[Non-uniform continuity of the semiflow map]\label{MT}
 Let $s>7/2$ be fixed.
Then, at any positive time  $t$, the  mapping 
 \[
  \V_s\ni u_0\mapsto u(t;u_0)\in\V_s
 \]
 associated to the semiflow defined by the evolution equation \eqref{PME}-\eqref{IC}  is real-analytic, but it is  not uniformly continuous.
 More precisely,  there exist  two sequences of positive solutions
  \[(u_n)_{n }, (v_n)_{n }\subset C([0,\infty),\V_s)\cap C^\infty([0,\infty)\times \s),\]
   and for each $T>0$ a positive constant $C>0$ with the following properties:
  \begin{equation}\label{A1}
  \begin{aligned}
   &\sup_{n\in\N}\max_{t\in[0,T]}\big(\|u_n(t)\|_{H^s}+\|v_n(t)\|_{H^s}\big)\leq C,\\
   &\lim_{n \to \infty} \|u_n(0)-v_n(0)\|_{H^s}=0,
  \end{aligned}
  \end{equation}
but
\begin{align}\label{A2}
   &\liminf_{n\to\infty} \inf_{t\in[\delta, T]}\|u_n(t)-v_n(t)\|_{H^s}\geq \sqrt{\pi}  \qquad\text{for each $\delta\in(0,T]$}.
  \end{align}
\end{thm}

Before starting our analysis we make some remarks.

\begin{rem}\label{R:1}
It is well-known that in the case of the linear heat equation
\begin{equation}\label{HE}
 u_t=u_{xx}, \qquad \text{$x\in\s$ and $t>0$}.
\end{equation}
the  estimate  $\|u(t; u_0)\|_{H^s}\leq \|u_0\|_{H^s}$ holds true for all $t\geq0,$ cf. e.g. \cite{HN01}.
Hence, the semiflow map   is  a contraction at any fixed positive time.
This suggests that the non-uniform continuity property derived in Theorem \ref{MT}
is due to the nonlinear character of the PME.
\end{rem}

\begin{rem}\label{R:2}
Let us point the construction presented in Theorem \ref{MT} has no correspondent within the setting of ODEs.
Indeed, we  note first that the right-hand side of \eqref{PME} corresponds to an operator $[u\mapsto f(u)]$ with 
$\|\p f(u)\|_{\kL(H^{s+\delta}, H^{s+\delta-2})}\leq K\|u\|_{H^s}$ for some positive constant $K$ and some $\delta\in(0,1/2)$ with $s+\delta-2\geq0.$\footnote{
To establish the well-posedness of the evolution problem \eqref{PME}-\eqref{IC} it is natural to exploit the quasilinear structure of \eqref{PME} and write $f(u)=-A(u)u$ with the operator
$A(v)u:=-(vu_x)_x $, $v\in \V_s$ and $u \in H^{s+\delta}(\s)$, satisfying $A\in C^\omega(\V_s,\mathcal{H}(H^{s+\delta}(\s), H^{s+\delta-2}(\s))), $ cf. e.g. \cite{Am93}.}
Now consider an autonomous ODE 
\[u'=f(u)\] 
with $E$ being a Banach space, $f\in \mbox{\it C}\, ^1(U,E)$, and $U\subset E$ open. 
Assuming that $\|\p f(u)\|\leq K\|u\|$ for all $u\in U$, and if
    $u_n , v_n :[0,T]\to  U$, $n\in\N$, are solutions   of $u'=f(u)$ satisfying \eqref{A1} (with $H^s(\s)$ replaced by $E$), then
one can easily see that 
\[
\|u_n (t)-v_n (t)\|\leq C(K,T)\|u_n (0)-v_n (0)\|\qquad\text{ for all $t\in[0,T]$, $n\in\N,$}
\]
so that \eqref{A2} can never be achieved.
\end{rem}

\paragraph{\bf Notation}

Throughout this paper, 
we shall denote by $C$ positive constants which may depend only upon $s$ and $T$. 
 Furthermore,   $ H^r(\s)$, with $r\in\mathbb{R},$ is the $L_2-$based Sobolev space on the unit circle.
Given $r\in\R$, we let  $\Lambda^r:=(1-\partial_x^2)^{r/2}$  denote  the Fourier multiplier with symbol $((1+|k|^2)^{r/2})_{k\in\Z},$
which is an isometric isomorphism $\Lambda^r:H^q(\s) \to H^{q-r}(\s) $   for all $q,r\in\R$.
Finally, we let $Q_T:=[0,T]\times\s.$

Before proceeding with the analysis, we recall the following Kato-Ponce  commutator estimate 
\begin{align}\label{KP}
\|[\Lambda^r, f]g\|_{L_2}\leq C_r\left(\| f_x\|_{L_\infty}\|\Lambda^{r-1}g\|_{L_2}+\|\Lambda^r f \|_{L_2}\|g\|_{L_\infty}\right)
\end{align}
which is valid for $r>3/2$ and all $f,g\in C^\infty(\s)$,  cf. \cite{KP88, Tay91}.
Hereby, $[\cdot ,\cdot ]$ is the   commutator defined by $[S,T]:=ST-TS.$ 

\paragraph{\bf Outline} 

In Section \ref{Sec:2} we establish first the well-posedness of the evolution problem \eqref{PME}-\eqref{IC}, and then 
we introduce two sequences of approximate solutions of the latter problem.
In Section \ref{Sec:3} we estimate the error between the approximate solutions and the exact solutions of \eqref{PME}-\eqref{IC}
associated to the approximate solutions in several Sobolev norms.
The proof of the main result Theorem \ref{MT} is presented in Section \ref{Sec:4}.

 \section{Well-posedness and approximate solutions for \eqref{PME}}\label{Sec:2}

The following well-posedness result is based on the theory of quasilinear parabolic problems as  presented in \cite{Am93}, 
cf. Theorem 12.1 and the discussion following Theorem 12.6.
That the strong solutions of \eqref{PME} are globally defined can be seen by arguing along the lines of Theorem 2.1 in \cite{ELM11} for example.

\begin{thm}[Global well-posedness]\label{T:1}
Let $s>3/2$  be given. 
Then, the problem \eqref{PME}-\eqref{IC} possesses for each $u_0\in \mathcal{V}_s$
a unique global solution $u(\cdot;u_0)$ within the class
 \[ C([0,\infty), \V_s)\cap C^\infty((0,\infty)\times\s).\]
 Moreover, the mapping
 \[
(t,u_0)\mapsto u(t;u_0)\in \V_s
 \]
defines a global semiflow which is   continuous  on $[0,\infty)\times \mathcal{V}_s$ and real-analytic on $(0,\infty)\times \mathcal{V}_s$. 
\end{thm}
 \begin{proof}
  The proof of Theorem \ref{T:1} follows from the results in \cite{Am93}.
 \end{proof}

In the remaining of this paper  $s>7/2$ and $T>0$ are kept fixed.
In a first step we construct two sequences of  approximate positive solutions of the equation \eqref{PME}.
The first solution sequence $(U_n)_n$ is defined by
\begin{equation} \label{AU}
    U_{n}(t,x):=  n ^{-3}+n^{-s}\cos(n x)
 \end{equation}
for    $(t,x)\in Q_T$ and   $n\in \N, \, n\geq2$.
 We note that  this solution is in fact independent of the  time variable and that for $n$  large the term involving the cosine has a high spatial frequency. 
 Moreover, these approximate solutions approach the boundary of $\V_s$ when  $n\to\infty$,   where the equation \eqref{PME} becomes  degenerate.
The second sequence $(V_n)_n$ of approximate solutions  is given by
\begin{equation} \label{AV}
    V_{n}(t,x):=  n ^{-1}+e^{-nt}n^{-s}\cos(nx)
 \end{equation}
 for    $(t,x)\in Q_T$ and   $n\in \N, \, n\geq2$.
The term involving the cosine has again high spatial frequency when $n$ is large, but now it decays very fast with respect to $tn$.
These solutions also approach the boundary of $\V_s$ when letting $n\to\infty.$

Before we estimate the error associated to  the approximate solutions, we note  that 
\begin{equation}\label{E}
  \|\cos (n x)\|_{H^r}=\sqrt{\pi} (1+n^2)^{r/2}
\end{equation}
for all $\sigma\in\R$ and $n\in \N, n \geq 2$. 
This property can be easily deduced from the definition of the spaces $H^r(\s)$, cf. e.g.  \cite{DGM13x}. 
The same arguments show also  that  $\|1\|_{H^{\sigma}}=\sqrt{2\pi}$.
It is now immediate to see that
\begin{align}
 \|U_n(0)-V_n(0)\|_{H^s}=&\sqrt{2\pi}(n^{-1}-n^{-3})\to_{n\to\infty}0,\label{PR1}
 \end{align}
 and that for all $t\geq 0$ we have
 \begin{align}
 \|U_n-V_n(t)\|_{H^s}\geq& (1-e^{-nt})\|n^{-s}\cos(nx)\|-\sqrt{2\pi}(n^{-1}-n^{-3})\nonumber\\
 =&\sqrt{\pi}(1-e^{-nt}) \Big(\frac{1+n^2}{n^2}\Big)^{s/2}-\sqrt{2\pi}(n^{-1}-n^{-3}).\label{PR2}
\end{align}
These properties are both related to \eqref{A1}-\eqref{A2}.
We derive now several properties for the approximate solutions sequences  $(U_n)_n$ and $(V_n)_n$.
It turns out that though the supremum norm  of $V_n$ is considerably larger than that of $U_n$, as $n^2\|U_n\|_{L_\infty}/\|V_n\|_{L_\infty}\to_{n\to\infty}1$, 
these approximate solutions satisfy surprisingly similar Sobolev estimates.

\begin{lemma}[Estimates for  the sequence $(U_n)_n$]\label{L:U1}
Let $s>7/2$ be given and $n\in \N, n\geq 2$ be arbitrary.
Then, we have 
\begin{align}\label{ABU1}
 &\frac{n^{-3}}{4}\leq U_n\leq 2n^{-3} \qquad\text{in  $\s,$}\\
 \label{ABU2}
&|\p_x U_n|\leq n^{-s+1} \qquad\text{in  $\s$.}
\end{align}
Additionally, there is a constant $C>0$ such that
\begin{align}\label{ABU3}
\|U_n\|_{H^r }\leq C n^{r-s} \qquad\text{for all $r\geq s $,}
\end{align}
and such the error term 
\begin{align}\label{EU}
E_{U_n}:=\p_tU_n-\p_x(U_n\p_xU_n)
\end{align}
satisfies
 \begin{align}
  &\|E_{U_n}\|_{H^1} \leq C n^{-s}.\label{EU1}
 \end{align}
\end{lemma} 

\begin{proof}
The estimates \eqref{ABU1}-\eqref{ABU3} are obvious consequences of the definition \eqref{AU} and of \eqref{E}.
In order to establish \eqref{EU1} we compute that
\begin{align*}
E_{U_n}=&-U_n\p_x^2U_n-(\p_x U_n)^2\\
=&n^{-s-1}\cos(nx)+n^{-2s+2}\cos(2nx),
\end{align*}
and therefore the desired conclusion \eqref{EU1} follows from \eqref{E}, since $s\geq 3.$
\end{proof}

Correspondingly, we have the following estimates for the sequence $(V_n)_n$.
\begin{lemma}[Estimates for  the sequence $(V_n)_n$]\label{L:V1}
Let $s>7/2$ be given and $n\in \N, n\geq 2$ be arbitrary.
Then, we have 
\begin{align}\label{ABV1}
& \frac{n^{-1}}{4}\leq V_n \leq 2n^{-1} \qquad\text{in $Q_T,$}\\
 \label{ABV2}
&|\p_x V_n|\leq n^{-s+1} \qquad\text{in $ Q_T$.}
\end{align}
Additionally, there is a constant $C>0$ such that 
\begin{align}\label{ABV3}
\max_{t\in[0,T]}\|V_n(t)\|_{H^r }\leq C n^{r-s} \qquad\text{for all  $r\geq s $,}
\end{align}
and such the error term 
\begin{align}\label{EV}
E_{V_n}:=\p_tV_n-\p_x(V_n\p_xV_n)
\end{align}
satisfies
 \begin{align}
  &\max_{t\in [0,T]}\|E_{V_n}(t)\|_{H^1} \leq C n^{-s} .\label{EV1}
 \end{align}
\end{lemma} 
\begin{proof}
 The estimates \eqref{ABV1}-\eqref{ABV3} are obtained  directly from \eqref{AV} and \eqref{E}.  
 Concerning \eqref{EV1}, we note that the terms with the highest norms cancel each other  
 \begin{align*}
  E_{V_n}=&\p_tV_n-V_n\p_x^2V_n-(\p_x V_n)^2\\
=&-ne^{-nt}n^{-s}\cos(nx)+ n^2e^{-nt}n^{-s}\cos(nx)(n^{-1}+e^{-nt}n^{-s}\cos(nx))\\
&-n^{ 2}e^{-2nt}n^{-2s}\sin^2(nx)\\
=&n^{-2s+2}e^{-2nt} \cos(2nx)
 \end{align*}
and the desired result follows from \eqref{E} and the fact that $s\geq 3$. 
\end{proof}

\section{Exact solutions and error estimates}\label{Sec:3}

The sequences $(u_n)$ and  $(v_n)_n$ from Theorem \ref{MT} are defined by letting $u_n$ and $v_n$ denote  the global solutions of \eqref{PME} which satisfy initially
\begin{equation}\label{u+v}
 u_n(0)=U_n\qquad\text{and}\qquad v_n(0)=V_n(0),
\end{equation}
respectively.
Because the initial data are smooth, the exact solutions
$u_n$ and $v_n$  share this property, that is $u_n, v_n\in C^\infty([0,\infty)\times\s)$.
Moreover, as $U_n$ and $V_n(0)$ are even, the uniqueness statement in the latter  theorem guarantees that 
$u_n(t)$ and $v_n(t)$ are even for all $t\geq0$.
Particularly, we have that
\begin{align}\label{BC}
 \p_xu_n(t,\pm \pi)=\p_x v_n(t,\pm \pi)=0
\end{align}
for all $t\geq 0$ and $n\in \N, n \geq 2.$
The parabolic maximum principles yield now that $u_n$ and $v_n$ satisfy  estimates similar to some of those 
 satisfied by $U_n$ and $V_n$.
Indeed, the strong maximum principle  and  Hopf's theorem applied in the cylinder $[-\pi,\pi]\times[0,T]$, cf. \cite{Li96}, and the relations \eqref{BC} yield that 
\begin{align}\label{ABu1}
 &\frac{n^{-3}}{4} \leq u_n \leq 2n^{-3}\qquad\text{ in $Q_T$,} \\
 \label{ABv1}
& \frac{n^{-1}}{4}\leq v_n \leq 2n^{-1} \qquad\text{in $Q_T$.}
\end{align}

On the other hand, differentiating \eqref{PME} with respect to $x$ we see  that $\p_xu_n$   is the solution of the problem
\begin{align*}
\left\{
\begin{array}{rlll}
 z_t&=&u_nz_{xx} +3\p_xu_nz_x  & \text{ for $t\geq 0$ and $x\in[-\pi,\pi]$,}\\[1ex]
 z(t,\pm \pi)&=&0& \text{ for $t\geq 0$,}\\[1ex]
 z(0,x)&=&\p_xU_n(x)& \text{ for  $x\in[-\pi,\pi]$.}
 \end{array}
\right.
\end{align*}
Hence, the strong maximum principle implies  
\begin{align}\label{DB1}
& |\p_xu_n |\leq n^{-s+1} \qquad\text{in $Q_T$.}
\end{align}
Using a  similar argument we obtain that
\begin{align}\label{DB2}
& |\p_xv_n |\leq n^{-s+1} \qquad\text{in $Q_T$.}
\end{align}
The estimates \eqref{ABu1}-\eqref{DB2} together with the Kato-Ponce commutator estimate \eqref{KP} are the key ingredient in the proof of the following lemma.

\begin{lemma}[The errors $\|U_n-u_n\|_{H^r}$ and $\|V_n-v_n\|_{H^r}$]\label{L:1} 
Let $r\geq s>7/2$ and let $n\in \N, n\geq 2$ be arbitrary.
Then,     there is a constant $C >0$ such that
 \begin{align}\label{DE_1}
 &\max_{t\in [0,T]}\|U_n-u_n(t)\|_{H^r}\leq C n^{r-s},\\
 \label{DE_2}
&\max_{t\in [0,T]}\|V_n(t)-v_n(t)\|_{H^r}\leq C n^{r-s}.
 \end{align}
\end{lemma}
\begin{proof}
In view of \eqref{ABU3}, it remains to show that
 \begin{equation*}
 \max_{t\in [0,T]}\|u_n(t)\|_{H^r}\leq C n^{r-s}.
 \end{equation*}
 Therefore, we compute that
 \begin{align*}
  \frac{1}{2}\frac{d}{dt}\|u_n\|_{H^r}^2=&\int_\s\Lambda^r u_n\Lambda^r \p_x(u_n\p_xu_n)\, dx=-\int_\s\Lambda^r (\p_xu_n)\Lambda^r (u_n\p_xu_n)\, dx\\
  =&  -\int _\s u_n|\Lambda^r (\p_xu_n)|^2\, dx-\int_\s \Lambda^r (\p_xu_n)[\Lambda^r, u_n] \p_xu_n\, dx
 \end{align*}
 in $[0,T].$
Using Young's inequality together with \eqref{KP}, \eqref{ABu1}, and \eqref{DB1} we then get
\begin{align*}
  \frac{1}{2}\frac{d}{dt}\|u_n\|_{H^r}^2\leq& - \frac{n^{-3}}{4}\|\Lambda^r (\p_xu_n)\|^2_{L_2}+\|\Lambda^r (\p_xu_n)\|_{L_2} \|[\Lambda^r, u_n] \p_xu_n\|_{L_2}\\
  \leq &-\frac{n^{-3}}{4}\|\Lambda^r (\p_xu_n)\|^2_{L_2}+C\|\Lambda^r (\p_xu_n)\|_{L_2} \| u_n\|_{H^r} \|\p_xu_n\|_{L_\infty}\\
   \leq &-\frac{n^{-3}}{4}\|\Lambda^r (\p_xu_n)\|^2_{L_2}+Cn^{-s+1}\|\Lambda^r (\p_xu_n)\|_{L_2} \| u_n\|_{H^r}  \\
   \leq &Cn^{-2s+5}\| u_n\|_{H^r}^2   
 \end{align*}
in $[0,T].$ 
 Since $s\geq5/2,$ we conclude that
 \[
 \max_{t\in [0,T]}\|u_n(t)\|_{H^r}\leq C \|u_n(0)\|_{H^r},
 \]
 and the desired claim  \eqref{DE_1} follows from \eqref{ABU3}.
 We still have to show that  \eqref{DE_2} holds true.
 This property follows by using the same arguments as in the proof of \eqref{DE_1}, because $v_n$  and $V_n$ satisfy similar estimates to 
 those verified  by $u_n$ and $U_n$, respectively (cf.  \eqref{ABV3}, \eqref{ABv1}, and  \eqref{DB2}).
\end{proof}

We remark that the estimate \eqref{EV1} can be improved in the sense that we may add a multiplicative term $e^{-2nt}$ on the right-hand side of \eqref{EV1}.
However, it
 is not clear from the proof above how to carry this property over  to the sequence $(v_n)_n.$ 
Based on Lemma \ref{L:1}, we obtain next estimates for the errors  in the $H^1-$norm.

\begin{lemma}[The errors $\|U_n-u_n\|_{H^1} $ and  $\|V_n-v_n\|_{H^1}$]\label{L:2}
Let  $s>7/2 $ be given and $n\in\N$ with $n\geq 2$ arbitrary.
Then, there exists a constant $C>0$ such that 
 \begin{align}\label{DEU}
&\max_{t\in[0,T]}\|U_n-u_n(t)\|_{H^1}\leq Cn^{-s},\\
\label{DEV}
&\max_{t\in[0,T]}\|V_n(t)-v_n(t)\|_{H^1}\leq Cn^{-s}.
 \end{align}
\end{lemma}

\begin{proof} 
Denoting by $w  $ the difference between the approximate solution $U_n$ and the associated exact solution $u_n$, that is  $w :=U_n-u_{n}$,
we see that $w$ is a  solution of the initial value problem
\begin{equation}\label{equ}
 \left\{
     \begin{array}{rlll}
           w_t& =& -(ww_x)_x+(U_n w_x)_x+(w\p_xU_{n })_x+E_{U_n}\qquad \text{for $t>0,$}\\[1ex]
           w(0)&=& 0,
     \end{array}
 \right.
\end{equation}
whereby $E_{U_n}$ is the error term defined by \eqref{EU}.
Using \eqref{equ}, we find that
 \begin{align*}
  \frac {1}{2}\frac{d}{dt}\|w \|_{H^1}^2=&\int_\s \Lambda w\Lambda w_t\, dx\\
  =&  -\int_\s \Lambda w\Lambda  (ww_x)_x \, dx
  +\int_\s \Lambda w\Lambda  (U_nw_x)_x \, dx+\int_\s \Lambda w\Lambda  (w\p_xU_n)_x \, dx\\
&
 +\int_\s\Lambda w \Lambda E_{U_n}\, dx\\
=&:I_1+I_2+I_3+I_4 
 \end{align*}
in $ [0,T].$
 We estimate  the terms $I_i$, $1\leq i\leq 4,$ separately.
Using integration by parts, the boundedness of   $(U_n-u_n)_n$ in $H^s(\s),$ cf. \eqref{DE_1}, and the fact that $H^s(\s)\hookrightarrow C^3(\s),$
we get  that
 \begin{align}
    I_1&=\int_\s  w_x\Lambda^2 (ww_x) \, dx=\int_\s w w^2_x\, dx-\int_\s w_x (ww_{x})_{xx}\, dx\nonumber\\
    &=\int_\s w w^2_x\, dx-3\int_\s w_{xx}w_x^2\, dx- \int_\s w_{xxx}ww_x\, dx\nonumber\\
    &\leq C\|w\|_{H^1}^2.\label{a1}
\end{align}
The same arguments, the fact that $U_n$ is positive and that $(U_n)_n$ is bounded in $H^s(\s),$ cf. \eqref{ABU1}  and \eqref{ABU3}, yield that
\begin{align} 
I_2=&-\int_\s   w_x\Lambda^2  (U_nw_x) \, dx=-\int_\s    U_nw_x^2 \, dx+\int_\s   w_x   (U_nw_x)_{xx} \, dx\nonumber\\
\leq & -\int_\s   U_nw_{xx} ^2 \, dx-\int_\s  \p_xU_nw_xw_{xx}  \, dx\leq \frac{1}{2}\int_\s \p_x^2 U_n w_x^2\, dx\nonumber\\
\leq &C\|w\|_{H^1}^2,\label{a2}
 \end{align}
and 
  \begin{align}
  I_3=&-\int_\s w_x\Lambda^2  (w\p_xU_n) \, dx=-\int_\s \p_xU_n ww_x\, dx-\int_\s w_x (w\p_xU_n)_{xx} \, dx\nonumber\\
  =&-\int_\s \p_xU_n ww_x\, dx+\int_\s w_{xx} (w\p_x^2U_n+w_x\p_xU_n) \, dx\nonumber\\
  =&-\int_\s \p_xU_n ww_x\, dx-\frac{1}{2}\int_\s\p_x^2U_n w_x^2-\int_\s\p_x^3U_nww_x\, dx-\int_\s\p_x^2U_n w_x^2\, dx\nonumber\\
   \leq &C\|w\|_{H^1}^2.\label{a3}
 \end{align}
 Finally, recalling \eqref{EU1}, we have
  \begin{align}
  I_4=\int_\s\Lambda w \Lambda E_{U_n}\, dx\leq \|\Lambda w\|_{L_2}\|E_{U_n}\|_{L_2}\leq  Cn^{-s}\|w\|_{H^1}.\label{a4}
 \end{align}
 Gathering \eqref{a1}-\eqref{a4}, we see that 
 \begin{align*}
 \frac{d}{dt}\|w\|_{H^1}^2 \leq & C\left(\|w\|_{H^1}^2+n^{-s}\|w\|_{H^1}\right)\qquad\text{in $[0,T],$}
\end{align*}
and therefore
    \begin{align*}
 \frac{d}{dt}\|w \|_{H^1}^2  \leq & C\left( \|w\|_{H^1}^2 +n^{-2s} \right) \qquad\text{in $[0,T]$}.
 \end{align*}
 Since $w(0)=0$, the desired estimate \eqref{DEU} follows   from  Gronwall's inequality.
  For the proof of \eqref{DEV} we argue analogously, as  the difference $V_n-v_n$ solves the same problem \eqref{equ}, but with $E_{U_n}$ replaced by $E_{V_n}$ (see. \eqref{EV}).
 Moreover, all the properties of  $E_{U_n}$, $U_n$ and $U_n-u_n$ that where used in the proof of \eqref{DEU}   are inherited   by  $E_{V_n}$, $V_n$ and $V_n-v_n$, 
 respectively, cf. Lemmas 
 \ref{L:U1}, \ref{L:V1}, and  \ref{L:1}.
\end{proof}

 
 \section{Proof of the main result} \label{Sec:4}
   In the remaining part we prove that the functions $u_n $ and $v_n $ 
   defined in the previous section, cf. \eqref{u+v}, satisfy  the properties required in Theorem \ref{MT} (for  $n\leq 1$ we simply set $u_n=v_n=1$).  
Indeed, the assertions \eqref{A1} follow directly from \eqref{PR1} and the Lemmas \ref{L:U1}, \ref{L:V1}, and \ref{L:1}.
In order to prove \eqref{A2}, we use the triangle inequality to get that 
\begin{align}
  \|u_n(t)-v_n(t)\|_{H^s}\geq& \|U_n-V_n(t)\|_{H^s} -\|U_n-u_n(t)\|_{H^s}-\|V_n(t)-v_n(t)\|_{H^s}\label{TI}
\end{align}
for all $t\in[0,T] $ and $n\geq 2.$ 
The first term is estimated from below as in \eqref{PR2}.
For the last  two terms we use the well-known interpolation inequality  
\begin{equation*}
      \| u \|_{H^s}\leq  \| u \|_{H^1}^{(r-s)/(r-1)}  \| u \|_{H^{r}}^{(s-1)/(r-1)}
\end{equation*}
with $r>s$ chosen arbitrarily, but fixed.
Combined with  \eqref{DE_1} and \eqref{DEU}, the latter inequality yields 
\begin{align} 
    \max_{t \in[0,T]}\|U_n-u_n(t)\|_{H^s}
        \leq & \|U_n-u_n(t)\|_{H^1}^{\frac{r-s}{r-1}}
                   \|U_n-u_n(t)\|_{H^r}^{\frac{s-1}{r-1}} \nonumber \\
        \leq  & C   n ^{\frac{-s(r-s)}{r-1}} n ^{\frac{(r-s)(s-1)}{r-1}}\nonumber \\
        \leq &Cn^\frac{s-r}{r-1},\label{BUB}
\end{align}
and similarly we obtain from \eqref{DE_2} and \eqref{DEV} that
\begin{align} 
     \max_{t \in[0,T]}\|V_n(t)-v_n(t)\|_{H^s} \leq Cn^\frac{s-r}{r-1}.\label{BVB}
\end{align}
Gathering \eqref{PR2}, \eqref{TI},  \eqref{BUB} and \eqref{BVB}, we have established   that  
\begin{align*}
    \|u_n(t)-v_n(t)\|_{H^s} 
        \geq &\sqrt{\pi}(1-e^{-nt}) \Big(\frac{1+n^2}{n^2}\Big)^{s/2}-\sqrt{2\pi}(n^{-1}-n^{-3})-Cn^\frac{s-r}{r-1}
\end{align*}
for all $t\in(0,T]$ and $n\geq 2$.
The desired final claim \eqref{A2} follows now  when letting $n\to\infty$.


\begin{thebibliography}{10}

\bibitem{Am93}
H.~Amann.
\newblock {Nonhomogeneous linear and quasilinear elliptic and parabolic
  boundary value problems}.
\newblock In {\em {Function spaces, differential operators and nonlinear
  analysis ({F}riedrichroda, 1992)}}, volume 133 of {\em {Teubner-Texte
  Math.}}, pages 9--126. Teubner, Stuttgart, 1993.

\bibitem{Bou03}
J.~Boussinesq.
\newblock {Recherches th{\'e}oriques sur l{\'e}coulement des nappes d'eau
  infiltr{\'e}s dans le sol et sur le d{\'e}bit de sources}.
\newblock {\em Comptes Rendus Acad. Sci. / J. Math. Pures Appl.}, 10:5--78,
  1903/1904.

\bibitem{CGT04}
J.~A. Carrillo, M.~P. Gualdani, and G.~Toscani.
\newblock {Finite speed of propagation in porous media by mass transportation
  methods}.
\newblock {\em C. R. Math. Acad. Sci. Paris}, 338(10):815--818, 2004.

\bibitem{CMV06}
J.~A. Carrillo, R.~J. McCann, and C.~Villani.
\newblock {Contractions in the 2-{W}asserstein length space and thermalization
  of granular media}.
\newblock {\em Arch. Ration. Mech. Anal.}, 179(2):217--263, 2006.

\bibitem{CT03}
J.~A. Carrillo and G.~Toscani.
\newblock {Wasserstein metric and large-time asymptotics of nonlinear diffusion
  equations}.
\newblock In {\em {New Trends in Mathematical Physics}}, pages 234--244. World
  Sci. Publ., Hackensack, NJ, 2004.

\bibitem{CE98}
A.~Constantin and J.~Escher.
\newblock {Well-posedness, global existence, and blowup phenomena for a
  periodic quasi-linear hyperbolic equation}.
\newblock {\em Comm. Pure Appl. Math.}, 51(5):475--504, 1998.

\bibitem{CL09}
A.~Constantin and D.~Lannes.
\newblock {The hydrodynamical relevance of the {C}amassa-{H}olm and
  {D}egasperis-{P}rocesi equations}.
\newblock {\em Arch. Ration. Mech. Anal.}, 192(1):165--186, 2009.

\bibitem{DGM13x}
N.~{Duruk Mutluba\c{s}}, A.~Geyer, and B.-V. Matioc.
\newblock {Non-uniform continuity of the flow map for an evolution equation
  modeling shallow water waves of moderate amplitude}.
\newblock {\em Nonlinear Anal. Real World Appl.}, 2013.
\newblock http://dx.doi.org/10.1016/j.nonrwa.2013.12.007.

\bibitem{ELM11}
J.~Escher, {\relax Ph}.~{Lauren\c cot}, and B.-V. Matioc.
\newblock {Existence and stability of weak solutions for a degenerate parabolic
  system modelling two-phase flows in porous media}.
\newblock {\em Ann. Inst. H. Poincar{\'e} Anal. Non Lin{\'e}aire},
  28(4):583--598, 2011.

\bibitem{EMM12}
J.~Escher, A.-V. Matioc, and B.-V. Matioc.
\newblock {Modelling and analysis of the {M}uskat problem for thin fluid
  layers}.
\newblock {\em J. Math. Fluid Mech.}, 14:267--277, 2012.

\bibitem{FL13}
Y.~Fu and Z.~Liu.
\newblock {Non-uniform dependence on initial data for the periodic modified
  {C}amassa-{H}olm equation}.
\newblock {\em NoDEA Nonlinear Differential Equations Appl.}, 20:741--755,
  2013.

\bibitem{Gr_B}
K.~Grayshan.
\newblock {Continuity properties of the data-to-solution map for the periodic
  {$b$}-family equation}.
\newblock {\em Differential Integral Equations}, 25(1-2):1--20, 2012.

\bibitem{Gr_N}
K.~Grayshan.
\newblock {Peakon solutions of the {N}ovikov equation and properties of the
  data-to-solution map}.
\newblock {\em J. Math. Anal. Appl.}, 397(2):515--521, 2013.

\bibitem{HK09}
A.~A. Himonas and C.~Kenig.
\newblock {Non-uniform dependence on initial data for the {CH} equation on the
  line}.
\newblock {\em Differential Integral Equations}, 22:201--224, 2009.

\bibitem{HKM10}
A.~A. Himonas, C.~Kenig, and G.~Misiolek.
\newblock {Non-uniform dependence for the periodic {CH} equation}.
\newblock {\em Commun. Partial Differential Equations}, 35:1145--1162, 2010.

\bibitem{HM10}
A.~A. Himonas and G.~Misiolek.
\newblock {Non-uniform dependence on initial data of solutions to the {E}uler
  {E}quations of hydrodynamics}.
\newblock {\em Comm. Math. Phys}, 296(1):285--301, 2010.

\bibitem{HN01}
J.~K. Hunter and B.~Nachtergaele.
\newblock {\em {Applied analysis}}.
\newblock World Scientific Publishing Co. Inc., River Edge, NJ, 2001.

\bibitem{KV88}
S.~Kamin and J.~L. V{\'a}zquez.
\newblock {Fundamental solutions and asymptotic behaviour for the
  {$p$}-{L}aplacian equation}.
\newblock {\em Rev. Mat. Iberoamericana}, 4(2):339--354, 1988.

\bibitem{Ka10}
D.~Karapetyan.
\newblock {Non-uniform dependence and well-posedness for the hyperelastic rod
  equation}.
\newblock {\em J. Differential Equations}, 249:796--826, 2010.

\bibitem{KP88}
T.~Kato and G.~Ponce.
\newblock {Commutator estimates and the euler and {N}avier-{S}tokes equations}.
\newblock {\em Comm. Pure Appl. Math.}, 41(7):891--907, 1988.

\bibitem{Li96}
G.~M. Lieberman.
\newblock {\em {Second order parabolic differential equations}}.
\newblock World Scientific Publishing Co. Inc., River Edge, NJ, 1996.

\bibitem{L95}
A.~Lunardi.
\newblock {\em {Analytic semigroups and optimal regularity in parabolic
  problems}}.
\newblock {Progress in Nonlinear Differential Equations and their Applications,
  16}. Birkh{\"a}user Verlag, Basel, 1995.

\bibitem{LPW13}
G.~Lv, P.~Y.~H. Pang, and M.~Wang.
\newblock {Non-uniform dependence on initial data for the $\mu-b$ equation}.
\newblock {\em Z. Angew. Math. Phys}, 64(5):1543--1554, 2013.

\bibitem{WL12}
G.~Lv and M.~Wang.
\newblock {Non-uniform dependence for a modified {C}amassa-{H}olm system}.
\newblock {\em J. Math. Physics}, 53:013101, 2012.

\bibitem{MK98}
H.~P. McKean.
\newblock {Breakdown of a shallow water equation}.
\newblock {\em Asian J. Math.}, 2(4):867--874, 1998.

\bibitem{Tay91}
M.~Taylor.
\newblock {\em {Pseudodifferential Operators and Nonlinear PDE}}.
\newblock Birkh{\"a}user, Boston, 1991.

\bibitem{V05}
J.~L. V{\'a}zquez.
\newblock {The porous medium equation. {N}ew contractivity results}.
\newblock In {\em {Elliptic and parabolic problems}}, volume~63 of {\em {Progr.
  Nonlinear Differential Equations Appl.}}, pages 433--451. Birkh{\"a}user,
  Basel, 2005.

\bibitem{Va07}
J.~L. V{\'a}zquez.
\newblock {\em {The Porous Medium Equation}}.
\newblock Clarendon Press, Oxford, 2007.

\end{thebibliography}
\end{document}